\documentclass[12pt,twoside,openany]{paper}
\usepackage{enumerate,amsmath,graphics,amssymb,graphicx,amscd,xypic,amscd,amsbsy,multirow,float,booktabs,verbatim}
\newcommand{\QED}{\hspace*{\fill}\rule{2.5mm}{2.5mm}}
\newtheorem{theorem}{Theorem}[section]

\newenvironment{proof}{\noindent{\bf Proof\ }}{\QED\\}
\newcommand{\R}{\mathbb{R}}

\newtheorem{lemma}{Lemma}[section]

\begin{document}
\begin{center}
\vspace{0.5cm} {\large \bf ``Quantiles symmetry"}\\
\vspace{1cm} Reza Hosseini, University of British Columbia\\
333-6356 Agricultural Road, Vancouver,\\
 BC, Canada, V6T1Z2\\
reza1317@gmail.com
\end{center}

\section{Abstract}

This paper finds a symmetry relation (between quantiles of a random
variable and its negative) that is intuitively appealing. We show
this symmetry is quite useful in finding new relations for
quantiles, in particular an equivariance property for quantiles
under continuous decreasing transformations.\\

\vspace{0.5cm}

 \noindent Keywords: Quantile function, distribution
function, symmetry, equivariance

\section{Introduction}

The traditional definition of quantiles for a random variable $X$
with distribution function $F$,
\[lq_X(p)=\inf \{x|F(x) \geq p\},\]
appears in classic works as \cite{parzen-1979}. We call this the
``left quantile function''. In some books (e.g. \cite{rychlik})
the quantile is defined as
\[rq_X(p)=\inf \{x|F(x) > p\}=\sup \{x|F(x) \leq p\},\]
this is what we call the ``right quantile function''. Also in
robustness literature people talk about the upper and lower medians
which are a very specific case of these definitions. Chapter 5 of
\cite{reza-phd} considers both definitions, explore their relation
and shows that considering both has several advantages. In
particular it provides a proof of the following lemma regarding the
properties of the quantiles.

\begin{lemma} (Quantile Properties Lemma) Suppose $X$ is a random
variable on the probability space $(\Omega,\Sigma,P)$ with
distribution function $F$:

\begin{enumerate}[a)]
\item $F(lq_F(p))\geq p$. \item $lq_F(p) \leq rq_F(p)$. \item
$p_1<p_2 \Rightarrow rq_F(p_1)\leq lq_F(p_2)$. \item
$rq_F(p)=\sup\{x|F(x)\leq p\}$. \item $P(lq_F(p)<X<rq_F(p))=0$.
i.e. $F$ is flat in the interval $(lq_F(p),rq_F(p))$. \item $P(X<
rq_F(p)) \leq p$. \item If $lq_F(p)<rq_F(p)$ then $F(lq_F(p))=p$
and hence $P(X\geq rq_F(p))=1-p$. \item
$lq_F(1)>-\infty,rq_F(0)<\infty$ and $P(rq_F(0) \leq X \leq
lq_F(1))=1$. \item $lq_F(p)$ and $rq_F(p)$ are non-decreasing
functions of $p$.  \item If $P(X=x)>0$ then $lq_F(F(x))=x.$ \item
$x<lq_F(p) \Rightarrow F(x)<p$ and $x>rq_F(p) \Rightarrow F(x)>p.$
\end{enumerate}

\label{quantile-properties}
\end{lemma}

Section \ref{section-quantile-symmetry} presents the desirable
``Quantile Symmetry Theorem'', a result that could be obtained
only by considering both left and right quantiles. This relation
can help us prove several other useful results regarding
quantiles. Also using the quantile symmetry theorem, we find a
relation for the equivariance property of quantiles under
non-increasing continuous transformations.

In order to motivate why this relation is intuitively appealing we
give the following example.
\begin{example}
A scientist asked two of his assistants to summarize the following
data regarding the acidity of rain:

{\tiny
\begin{table}[H]
  \centering  \footnotesize
  \begin{tabular}{lcccccccc}
\toprule[1pt]
row number & $pH$ & $aH$\\
\midrule[1pt]
1 & 4.7336  & $18.4672 \times 10^{-6}$\\
2 & 4.8327 & $14.6994 \times 10^{-6}$\\
3 & 4.8492 & $14.1514 \times 10^{-6}$\\
4 & 5.0050  & $9.8855 \times 10^{-6}$\\
5 & 5.0389  & $9.1432 \times 10^{-6}$\\
6 & 5.2487  & $5.6403 \times 10^{-6}$\\
7 & 5.2713  & $5.3543 \times 10^{-6}$\\
8 & 5.2901  & $5.1274 \times 10^{-6}$\\
9 & 5.5731  & $2.6724  \times 10^{-6}$\\
10 & 5.6105 & $2.4519 \times 10^{-6}$\\
\bottomrule[1pt]
\end{tabular}
\caption{Rain acidity data}
 \label{table:q-def-acidity}
\end{table}
}

$pH$ is defined as the cologarithm of the activity of dissolved
hydrogen ions $(H^+)$.

\[pH=-\log_{10}aH.\]
In the data file handed to the assistants (Table
\ref{table:q-def-acidity}) the data is sorted with respect to $pH$
in increasing order from top to bottom. Hence the data is arranged
decreasingly with respect to $aH$ from top to bottom.

The scientist asked the two assistants to compute the 20th and
80th percentile of the data to get an idea of the variability of
the acidity. The first assistant used the $pH$ scale and the
traditional definition of the quantile function:

\[q_F(p)=\inf \{x|\; F(x) \geq p\},\]
where $F$ is the empirical distribution of the data. He obtained
the following numbers

\begin{equation}q_F(0.2)=4.8327\; \mbox{and} \;q_F(0.8)=5.2901,
\label{eq-first-assistant-res}\end{equation} which are positioned
in row 2 and 8 respectively.

The second assistant also used the traditional definition of the
quantile function and the $aH$ scale to get

\begin{equation}q_F(0.2)=2.6724 \times 10^{-6}\; \mbox{and} \; q_F(0.8)=14.1514
\times 10^{-6}, \label{eq-second-assistant-res}\end{equation}
which correspond to row 9 and 3.

The scientist noticed that the assistants had used different
scales. Then he thought since one of the scales is in the opposite
order of the other and 0.2 and 0.8 have the same distance from 0
and 1 respectively, he must get the other first assistant's result
by transforming the second's. So he transformed the second
assistant's results given in  Equation
\ref{eq-second-assistant-res} (or by simply looking at the
corresponding rows, 9 and 3 under $pH$), to get

\[ 5.5731\; \mbox{and} \; 4.8492,\]
which is not the same as the first assistant's result in Equation
\ref{eq-first-assistant-res}. He noticed that the position of
these values are off by only one position from the previous values
(being in row 9 and 3 instead of 8 and 2).

Then he tried the same himself for 25th and 75th percentile using
both scales

\[pH: q_F(0.25)=4.8492\; \mbox{and} \; q_F(0.75)=5.2901,\]
which are positioned at 3rd and 8th row.

\[aH: q_F(0.25)=5.1274 \times 10^{-6}\; \mbox{and} \;q_F(0.25)=14.1514 \times 10^{-6},\]
which are also positioned at 8th and 3rd row. This time he was
surprised to observe the symmetry he expected. He wondered when
such symmetry exist and what can be said in general. He
conjectured that the asymmetric definition of the traditional
quantile is the reason of this asymmetry. He also conjectured that
the symmetry property is off at most by one position in the
dataset.
\end{example}

\section{Quantile symmetries}
\label{section-quantile-symmetry}

This section studies the symmetry properties of distribution
functions and quantile functions. Symmetry is in the sense that if
$X$ is a random variable, some sort of symmetry should hold
between the quantile functions of $X$ and $-X$. We only treat the
quantile functions for distributions here but the results can
readily be applied to data vectors by considering their empirical
distribution function.

Here we consider different forms of distribution functions. The
usual one is defined to be $F^{c}_X(x)=P(X \leq x)$. But clearly
one can also consider $F^{o}_X(x)=P(X <x)$, $G^{c}_X(x)=P(X\geq
x)$ or $G^{o}_X(x)=P(X>x)$ to characterize the distribution of a
random variable. We call $F^{c}$ the left-closed distribution
function, $F^{o}$ the left-open distribution function, $G^{c}$ the
right-closed and $G^{o}$ the right-open distribution function.
Like the usual distribution function  these functions can be
characterized by their limits in infinity, monotonicity and
one-sided continuity.

First note that

\[F^c_{-X}(x)=P(-X \leq x)=P(X \geq -x)=G_X^c(-x).\]
Since the left hand side is right continuous, $G^c_X$ is left
continuous. Also note that
\begin{align*}
F^c_X(x)+G^o_X(x)=1 \Rightarrow G^o_X(x)=1-F^c_X(x),\\
F^o_X(x)+G^c_X(x)=1 \Rightarrow F^o_X(x)=1-G^c_X(x).\\
\end{align*}
The above equations imply the following:\\
a) $G^o$ and $F^c$ are right continuous.\\
b) $F^{o}$ and $G^c$ are left continuous.\\
c) $G^o$ and $G^c$ are non-decreasing.\\
d) $lim_{x \rightarrow \infty} F(x)=1$ and $lim_{x \rightarrow
-\infty} F(x)=0$ for $F=F^o,F^c.$\\
e) $lim_{x \rightarrow \infty} G(x)=0$ and $lim_{x \rightarrow
-\infty} G(x)=1$ for $G=G^o,G^c$.\\

It is easy to see that the above given properties for
$F^{o},G^o,G^c$ characterize all such functions. The proof can be
given directly using the properties of the probability measure
(such as continuity) or by using arguments similar to the above.

Another lemma about the relation of $F^{c},F^{o},G^o,G^c$ is given
below.

\begin{lemma} Suppose $F^o,F^c,G^o,G^c$ are defined as above. Then\\
a) if any of $F^{c},F^{o},G^o,G^c$ are continuous, all the other
ones are
continuous too.\\
b) $F^c$ being strictly increasing is equivalent to $F^o$ being
strictly
increasing.\\
c)  $F^c$ being strictly increasing is equivalent to $G^o$ being strictly decreasing.\\
d) $G^c$ being strictly decreasing is equivalent to $G^o$ being
strictly increasing.
\end{lemma}

\begin{proof}
a) Note that \[\lim_{y \rightarrow x^{-}} F^c(x)=\lim_{y
\rightarrow x^{-}} F^o(x),\] and \[\lim_{y \rightarrow x^{+}}
F^c(x)=\lim_{y \rightarrow x^{+}} F^o(x).\] If these two limits
are equal for either $F^c$ or
$F^o$ they are equal for the others as well.\\
b) If either $F^c$ or $F^o$ are not strictly increasing then they
are constant on $[x_1,x_2],x_1<x_2$. Take $x_1<y_1<y_2<x_2.$ Then

\[F^o(x_1)=F^o(x_2) \Rightarrow P(y_1\leq X \leq y_2)=0 \Rightarrow F^c(y_1)=F^c(y_2).\]
Also we have
\[F^c(x_1)=F^c(x_2) \Rightarrow P(y_1\leq X \leq y_2)=0 \Rightarrow F^o(y_1)=F^o(y_2). \]
\\
c) This is trivial since $G^o=1-F^c$.\\
d) If $G^c$ is strictly decreasing then $F^o$ is strictly
increasing since $G^c=1-F^o$. By Part b), $F^c$ strictly is
increasing. Hence $G^o=1-F^c$ is strictly decreasing.
\end{proof}

The relationship between these distribution functions and the
quantile functions are interesting and have interesting
implications. It turns out that we can replace  $F^c$ by $F^o$ in
some definitions.

\begin{lemma} Suppose $X$ is a random variable with open and closed
left distributions $F^o,F^c$ as well as open and closed right
distribution functions $G^o,G^c$. Then\\
 a) $lq_X(p)=\inf \{x|F^o_X(x)\geq
p\}$. In other words, we can replace $F^c$ by
$F^o$ in the left quantile definition.\\
b) $rq_X(p)=\inf \{x|F^o_X(x)> p\}$. In other words, we can
replace $F^c$ by
$F^o$ in the right quantile definition.\\
\end{lemma}

\begin{proof}
a) Let $A= \{x|F^o_X(x)\geq p\}$ and $B= \{x|F^c_X(x)\geq p\}$. We
want to show that $\inf A=\inf B$. Now\\
\[A \subset B \Rightarrow \inf A \geq \inf B.\]
But \[\inf B < \inf A \Rightarrow \exists x_0,y_0,\;\; \inf B
<x_0<y_0 < \inf A.\] Then

\begin{align*}\inf B < x_0 \Rightarrow \exists b
\in B,\;b<x_0 \Rightarrow \exists b \in \R,\; p \leq P(X \leq
b)\leq
P(X \leq x_0)\\
\Rightarrow P(X \leq x_0)\geq p \Rightarrow P(X <y_0)\geq p.
\end{align*}
On the other hand

\[y_0 < \inf A \Rightarrow y_0 \notin A \Rightarrow P(X<y_0)<p,\]
which is a contradiction, thus proving a).\\
 b) Let $A= \{x|F^o_X(x)> p\}$ and $B=
\{x|F^c_X(x)> p\}$. We
want to show $\inf A=\inf B$. Again,\\
\[A \subset B \Rightarrow \inf A \geq \inf B.\]
But
\[\inf B < \inf A \Rightarrow \exists x_0,y_0,\; \inf B<x_0<y_0<\inf A.\] Then

\begin{align*}\inf B < x_0 \Rightarrow \exists b
\in B,\;b<x_0 \Rightarrow \exists b \in \R,\; p < P(X \leq b)\leq
P(X \leq x_0)\\
\Rightarrow P(X \leq x_0)> p \Rightarrow P(X <y_0)> p.
\end{align*}
On the other hand,
\[y_0 < \inf A \Rightarrow y_0 \notin A \Rightarrow P(X<y_0) \leq p,\]
which is a contradiction.
\end{proof}

Using the above results, we establish the main theorem of this
section which states the symmetry property of the left and right
quantiles.

\begin{theorem}(Quantile Symmetry Theorem) Suppose $X$ is a random variable and $p \in [0,1]$.
Then

\[lq_X(p)=-rq_{-X}(1-p).\]

\label{theo-quantile-symmetry}
\end{theorem}

\noindent {\bf Remark.} We immediately conclude
\[rq_X(p)=-lq_{-X}(1-p),\]
by replacing $X$ by $-X$ and $p$ by $1-p$.

\begin{proof}

\begin{align*}
R.H.S=-\sup\{x|P(-X \leq x) \leq 1-p\}&=&\\
\inf\{-x|P(X \geq -x)\leq 1-p\}&=&\\
\inf\{x|P(X \geq x)\leq 1-p\}&=&\\
\inf\{x|1-P(X \geq x) \geq p\}&=&\\
\inf\{x|1-G^{c}(x) \geq p\}&=&\\
\inf\{x|F^{o}(x)\geq p\}=lq_X(p).&&
\end{align*}
\end{proof}

Now we show how these symmetries can become useful to derive other
relationships for quantiles.

\begin{lemma} Suppose $X$ is a random variable with distribution
function $F$. Then
\[lq_X(p)=\sup\{x|F^{c}(x)<p\}.\]
\end{lemma}

\begin{proof}
\begin{align*}
lq_X(p)=-rq_{-X}(1-p)=-\inf\{x|F^o_{-X}(x)>1-p\}&=&\\
-\inf\{x|1-G^{c}_{-X}(x)>1-p\}=\sup\{-x|G_{-X}^c(x)<p\}&=&\\
\sup \{-x|P(-X\geq x)<p\}=\sup\{x|P(X \leq x)<p\}&=&\\
\sup\{x|F^{c}(x)<p\}&.&
\end{align*}
\end{proof}

\section{Equivariance of quantiles under decreasing transformations}

It is widely claimed that (e.g. in
\cite{quantile-regression-koenker} or
\cite{quantile-regression-hao}) the traditional quantile function
is equivariant under monotonic transformations.  \cite{reza-phd}
shows that this does not hold even for strictly increasing
functions. However he proves that the traditional quantile
function is equivariant under non-decreasing left continuous
transformations. He also shows that the right quantile function is
equivariant under non-decreasing right continuous transformations.
In other words

\[lq_{\phi(X)}(p)=\phi(lq_X(p)),\]
where $\phi$ is non-decreasing left continuous. Also
\[rq_{\phi(X)}(p)=\phi(rq_X(p)),\]
for $\phi:\R \rightarrow \R$ non-decreasing right continuous.

Using the quantile symmetry, a similar neat result is found for
continuous decreasing transformations using the Quantile Symmetry
Theorem.

\begin{theorem} (Decreasing transformation equivariance)\\
a) Suppose $\phi$ is non-increasing and  right continuous on $\R$.
Then
\[lq_{\phi(X)}(p)=\phi(rq_X(1-p)).\]
b) Suppose $\phi$ is non-increasing and  left continuous on $\R$.
Then
\[rq_{\phi(X)}(p)=\phi(lq_X(1-p)).\]
 \label{theo-decreas-equiv}
\end{theorem}

\begin{proof}
a) By the Quantile Symmetry Theorem, we have
\[lq_{\phi(X)}(p)=-rq_{-\phi(X)}(1-p).\]
But $-\phi$ is  non-decreasing right continuous, hence the above
is equivalent to
\[-(-\phi(rq_X(1-p)))=\phi(rq_X(1-p)).\]
b) By the Quantile symmetry Theorem
\[rq_{\phi(X)}(p)=-lq_{-\phi(X)(1-p)}=-(-\phi(lq_X(1-p)))=\phi(lq_X(p)),\]
since $-\phi$ is non-decreasing and left continuous.
\end{proof}

\bibliographystyle{plain}
\bibliography{D:/School/Research/PhD_thesis/mybibreza}

\begin{thebibliography}{1}

\bibitem{quantile-regression-hao}
L.~Hao and D.~Q. Naiman.
\newblock {\em Quantile Regression}.
\newblock Quantitative Applications in the Social Sciences Series. SAGE
  publications, 2007.

\bibitem{reza-phd}
R.~Hosseini.
\newblock {\em Statistical Models for Agroclimate Risk Analysis}.
\newblock PhD thesis, Department of Statistics, UBC, 2009.

\bibitem{quantile-regression-koenker}
R.~Koenker.
\newblock {\em Quantile Regression}.
\newblock Cambridge university press, 2005.

\bibitem{parzen-1979}
E.~Parzen.
\newblock Nonparametric statistical data modeling.
\newblock {\em Journal of the American Statistical Association}, 74:105--121,
  1979.

\bibitem{rychlik}
T.~Rychlik.
\newblock {\em Projecting statistical functionals}.
\newblock Springer, 2001.

\end{thebibliography}
\end{document}